\documentclass[letterpaper, 10pt, conference]{ieeeconf}

\IEEEoverridecommandlockouts 

\usepackage{todonotes}

\usepackage[utf8]{inputenc}
\usepackage{latexsym,amssymb,amsmath, amsbsy,amsopn, amstext}
\usepackage{graphicx,epsfig,tikz,pgf,hyperref,cancel,color,float,ifpdf}
\usepackage{extarrows,mathrsfs}
\usepackage{amsmath,cite,amssymb,stfloats}
\usepackage{tikz}
\usepackage{subcaption} 
\newif\ifdraft
\drafttrue

\graphicspath{{./figures/}}
\newcommand{\R}{\mathbb{R}}
\newcommand{\IR}{\mathbb{IR}}
\newcommand{\Z}{\mathbb{Z}}

\newcommand{\x}{\mathbf{x}}
\newcommand{\ub}{\mathbf{u}}

\newcommand{\C}{\mathcal{C}}

\newcommand{\Kzcbf}{K_{\mathrm{zcbf}}}

\newcommand*\diff{\mathop{}\!\mathrm{d}}

\newcommand{\ubold}{{\bf u}}

\newcommand{\xbar}{x}

\newcommand{\la}{\lambda}

\newtheorem{definition}{\bfseries Definition}
\newtheorem{proposition}{\bfseries Proposition}
\newtheorem{example}{\bfseries Example}
\newtheorem{assumption}{\it Assumption}

\newtheorem{corollary}{\bfseries Corollary}

\newtheorem{remark}{\bfseries Remark}

\newtheorem{problem}{\bfseries Problem}

\allowdisplaybreaks[4]

\title{\LARGE \bf  Control Barrier Function Meets Interval Analysis: Safety-Critical Control with Measurement and Actuation Uncertainties}

\author{Yuhao Zhang, Sequoyah Walters, Xiangru Xu
\thanks{Y. Zhang, S. Walters and X. Xu are with the Department of Mechanical Engineering, University of Wisconsin-Madison, Madison, WI, USA. Email: {\tt\small \{yuhao.zhang2,snwalters2,xiangru.xu\}@wisc.edu}.}
}

\begin{document}

\maketitle

\begin{abstract}
This paper presents a framework for designing provably safe feedback controllers for sampled-data control affine systems with measurement and actuation uncertainties. Based on the interval Taylor model of nonlinear functions, a sampled-data control barrier function (CBF) condition is proposed which ensures the forward invariance of a safe set for sampled-data systems. Reachable set overapproximation and Lasserre's hierarchy of polynomial optimization are used for finding the margin term in the sampled-data CBF condition. Sufficient conditions for a safe controller in the presence of measurement and actuation uncertainties are proposed, for CBFs with relative degree 1 and higher relative degree individually. The effectiveness of the proposed method is illustrated by two numerical examples and an experimental example that implements the proposed controller on the Crazyflie quadcopter in real-time.
\end{abstract}

\section{Introduction}\label{sec:intro}


Designing feedback controllers  that enforce safety
specifications is a recurring challenge in many real
systems such as automotive and robotic systems. 
Safety conditions are normally specified in terms of forward invariance of a set, which can be established through the barrier function (or barrier certificate) without finding trajectories of a system \cite{blanchini2015book,aubin2009viability,prajna2007framework,prajna2004safety}. For control systems, controlled invariant sets are used to encode the correct behavior of the closed-loop system and characterize a set of feedback control laws that will achieve it \cite{wolff2005invariance,kimmel2017invariance}. 
Inspired by the automotive safety-control problems,  \cite{aaron2014cbf,Xu2015ADHS,aaron2016barriertac} proposed reciprocal and zeroing control barrier functions (CBFs) which extend previous barrier conditions to only requiring a single sub-level set controlled invariant.
Families of control policies that guarantee safety can be obtained by solving a convex quadratic program  (QP). This CBF-QP framework has been used in applications such as automotive safety systems, bipedal robots, quadcopters, robotic manipulators, and multi-agent systems 
\cite{xu2017correctness,Hsu15Backstepping,nguyen2016optimal,wang2018safe,cortez2021control,pickem2017robotarium}. 

CBFs proposed in \cite{aaron2014cbf,Xu2015ADHS,aaron2016barriertac} provide forward invariance guarantees for the safety set in the continuous-time sense, but require  the control law being updated  continuously. This requirement is difficult to realize in practice because most controllers are digitally implemented in a sampled-data fashion and the time to solve the convex QP is not negligible in safety-critical applications. Therefore, new conditions are needed to ensure the forward invariance for sampled-data systems with piecewise-constant controllers (also called zero-order-hold controllers). CBFs for sampled-data systems have been investigated in \cite{cortez2021control,ghaffari2018safety,yang2019self,gurriet2019realizable,singletary2020control,breeden2021control}. These existing works are either designed only for a specific type of systems or based on non-convex optimization which is not suitable for real-time control implementations. On the other hand, almost all the existing results using CBFs rely on accurate state and actuation information, which is difficult to obtain in practice. In \cite{dean2020guaranteeing}, A measurement-robust CBF was proposed for the safety of learned perception modules; in  \cite{takano2018robust}, an unscented Kalman filter is integrated with CBF-QP to attenuate the effects of state disturbances and measurement noises. In spite of these interesting results, a systematic approach to handle measurement and actuation uncertainties for the CBF-based safe controller is still lacking, especially one that is suitable for real-time applications.

This paper presents a safety control design framework for sampled-data systems in the presence of measurement and actuation uncertainties, by leveraging tools from interval analysis and CBFs. The contributions  are at least threefold: (i) Based on the interval Taylor model of nonlinear functions, a sampled-data control barrier function (SDCBF) condition is proposed to guarantee the forward invariance of a safe set for sampled-data systems. CBFs with a relative degree 1 and a higher relative degree are considered individually. (ii) Sufficient conditions for CBF-based sampled-data safe controller in the presence of inaccurate state measurement and actuation are proposed. (iii) Efficient algorithms are proposed to compute the new SDCBF conditions utilizing interval arithmetic and polynomial optimization techniques, and implemented on the Crazyflie quadcopter hardware.

The remainder of the paper is laid out as follows: Section \ref{sec:moti} introduces preliminaries on CBFs and interval arithmetic, and presents the problem statement. Section \ref{range_bound} introduces the SDCBF condition and a framework to compute the condition  efficiently. Section \ref{uncertain} presents a CBF-based safe controller for sampled-data systems in the presence of inaccurate state measurement and actuation. Two numerical simulations and a quadcopter experiment  are shown in Section \ref{simulation} before the paper is concluded in Section \ref{conclude}.

\section{Preliminaries and Problem Statement}\label{sec:moti}

\subsection{Control Barrier Functions}

Consider a control affine system 
\begin{eqnarray}
\label{eqnsys}
\dot{x} = F(x,u):=f(x) + g(x) u,
\end{eqnarray}
where $x \in \R^n$, $u \in U \subset \R^m$, and $f:\R^n\rightarrow \R^n$ and $g:\R^n\rightarrow \R^m$ are locally Lipschitz continuous. Given a control signal $u(t)$, the solution of \eqref{eqnsys} at time $t$ with initial condition $x_0\in\mathbb{R}^n$ at time $t_0$ is denoted by $x(t;t_0,x_0)$ or simply $x(t)$ when $t_0,x_0$ are clear from context. For simplicity,  assume that the solution of  \eqref{eqnsys} exists for all $t\geq t_0$. A set $\mathcal{S}$ is called (forward) controlled invariant with respect to system \eqref{eqnsys} if for every $x_0 \in \mathcal{S}$, there exists a control signal $u(t)$ such that $x(t;t_0,x_0) \in \mathcal{S}$ for all $t \geq t_0$. The set $\mathcal{S}$ is called safe if it is controlled invariant.

The (forward) \emph{reachable set} of system \eqref{eqnsys} from an initial set $\mathcal{X}_0 \subset \mathbb{R}^n$ at time $t>t_0$ is defined as $\mathcal{R}(t,\mathcal{X}_0)\triangleq\{x(t;t_0,x_0)\mid x_0\in\mathcal{X}_0,u\in U\}$. The (forward) \emph{reachable tube} of system \eqref{eqnsys} from an initial set $\mathcal{X}_0$ over a time interval $[t_1,t_2]$ where $t_2>t_1\geq t_0$ is $\mathcal{R}([t_1,t_2],\mathcal{X}_0)\triangleq\{x(t;t_0,x_0)\mid x_0\in\mathcal{X}_0,t\in[t_1,t_2],u\in U\}.$   

Consider a safe set $\mathcal{C} \subset \R^n$ defined by 
\begin{equation}\label{setc}
    \mathcal{C}= \{ x \in \R^n : h(x) \geq 0\}
\end{equation}
for a continuously differentiable function $h: \R^n \to \R$. A continuous function $\alpha: (-b,a) \rightarrow (-\infty,\infty)$ for some $a, b > 0$ belongs to \emph{extended class} $\mathcal{K}$ if it is strictly increasing and $\alpha(0)=0$. Given a function $h$ with a relative degree 1, it is called a zeroing CBF if  there exists an extended class $\mathcal{K}$ function $\alpha$ such that
\begin{align}\label{ineqZCBF}
& \sup_{u \in U}  \left[ L_f h(x) + L_g h(x) u + \alpha(h(x))\right] \geq 0
\end{align}
where $L_fh(x)=\frac{\partial h}{\partial x}f(x)$ and $L_gh(x)=\frac{\partial h}{\partial x}g(x)$ are Lie derivatives \cite{Xu2015ADHS}. In this paper, we will assume $\alpha(h(x)) = \gamma h(x)$ for a positive constant $\gamma$. However, results of this paper can be naturally extended to any continuously differentiable extended class $\mathcal{K}$ function $\alpha$. 
The set of control inputs that satisfy \eqref{ineqZCBF} for all $x\in\R^n$ is defined as 
\begin{align}
\Kzcbf(x) \!=\!  \{ u \in U\mid L_f h(x) \!+\! L_g h(x) u \!+\! \gamma h(x) \geq 0\}.    
\end{align}
It was proven in \cite{Xu2015ADHS} that any locally Lipschitz continuous controller $u(x) \in \Kzcbf(x)$ for every $x\in\R^n$ will guarantee the forward invariance (or safety) of $\mathcal{C}$. The safe control law is obtained by solving the following convex CBF-QP:
\begin{align}
u^*(\xbar) & =   \underset{u\in U}{\operatorname{argmin}}   \quad \|u-u_{nom}\|_2 \label{eqn:QPCLFCBF}\tag{CBF-QP}\\
\mathrm{s.t.} & \; L_f h(\xbar) + L_g h(\xbar) u + \gamma h(\xbar)\geq 0,\;\forall x\in\R^n,\nonumber
\end{align}
where $u_{nom}$ is a nominal  controller that is potentially unsafe.

Given a function $h(x)$ that is $C^r (r\geq 2)$  and has a relative degree $r$, it is called a zeroing CBF if  there exists a column vector ${\bf a}\in\R^r$ such that $\forall x\in\R^n$,
\begin{align}\label{ineq:ZCBF2}
	& \sup_{u \in U}  [L_g L_f^{r-1}h(x)u +L_f^rh(x)+{\bf a}^{\top}\eta(x) ] \geq 0
\end{align}
where $\eta(x)=[L_f^{r-1}h, L_f^{r-2}h,...,h]^{\top}\in\R^r$, and ${\bf a}=[a_1,...,a_r]^{\top}\in\R^r$ is chosen such that the roots of $p_0^r(\la)=\la^r+a_1\la^{r-1}+...+a_{r-1}\la+a_r$ are all negative reals $-\la_1,...,-\la_r<0$. 
Define functions $s_k(x(t))$ for $k=0,1,...,r-1$ as follows:
\begin{align}\label{liftedh}
s_0(x(t))&=h(x(t)),\;s_{k}(x(t))=(\frac{\diff}{\diff t}+\la_k)s_{k-1}.
\end{align}
It was shown in \cite{nguyen2016exponential} that if $s_k(x(0))\geq 0$ for $k=0,1,...,r-1$, then any controller $u(x) \in \{ u \in U : L_g L_f^{r-1}h(x)u +L_f^rh(x)+{\bf a}^{\top}\eta(x) \geq 0\}$ that is locally Lipschitz  will guarantee the forward invariance of $\mathcal{C}$. The safe controller is obtained by solving a QP similar to \eqref{eqn:QPCLFCBF}.

\subsection{Interval Arithmetic}

A real interval $[a]=[\underline{a},\bar a]$ is a subset of $\R$. 
The set of all real intervals of $\R$ is denoted as $\IR$. The set of n-dimensional real interval vectors is denoted by $\IR^n$. 
Real arithmetic operations on $\R$ can be extended to $\IR$ as follows \cite{moore1966interval}: for $\circ \in\{+,-,*,\div\}$,
\begin{align*}
[a]\circ[b]=\{\inf_{x\in[a],y\in[b]}x\circ y, \sup_{x\in[a],y\in[b]}x\circ y\}.
\end{align*}
Classical operations for interval vectors
 are direct extensions of the same operations for real vectors  \cite{jaulin2001interval,moore2009introduction}.


\subsection{Problem Statement}
There exists a gap between the theoretical safety guarantee provided by \eqref{eqn:QPCLFCBF} and the constraint satisfaction in real control implementations. The safety guarantee provided by the control input $u^*(x)$ from \eqref{eqn:QPCLFCBF} is predicated on the following implicit ``assumptions'': (1) the time to solve the QP is negligible so that the controller can be updated continuously; (2) the accurate state information is known; (3) the actuator is perfect such that the exact control input generated by  \eqref{eqn:QPCLFCBF} is implemented by the actuator. 
However, these ``assumptions'' can hardly be satisfied in reality:  modern systems are predominantly based on digital electronics, which means that the input can only be updated at discrete time instances; the state information is usually obtained from sensors and  contaminated with unknown noise; the desired control command is not perfectly achievable by real-life  actuators. 


For a sampled-data system, the sampling instants are described by a
sequence of strictly increasing positive real numbers $\{t_k\}$, $k\in\Z_{\geq 0}$, where $t_0=0$, $t_{k+1}-t_k>0$, $\lim_{k\rightarrow\infty}t_k=\infty.$ 
Define the sampling interval between $t_k$ and $t_{k+1}$ as 
$$
\Delta_k=t_{k+1}-t_k.
$$
The sampling mechanism is called a periodic sampling if $\Delta_k$ are the same for all $k$, and an aperiodic sampling  otherwise. 
At each sampling instance $t_k$, the state and input of the system are denoted as $x_k = x(t_k)$ and $u_k = u(t_k)$, respectively. The control input $u(t)$ is assumed to be a piecewise constant signal with respect to $\{t_k\}$, i.e.,
\begin{align}
u(t)=u_k,\;\forall t\in[t_k,t_{k+1}).\label{eqinput}
\end{align} 
At each sampling time $t_k$, the control input $u_k$ is chosen from the set $\Kzcbf(x_k)$, i.e., 
\begin{align}
L_f h(x_k) + L_g h(x_k) u_k + \gamma h(x_k) \geq 0.\label{eqsampled}
\end{align}
The CBF condition in \eqref{eqsampled} may not be satisfied during inter-sampling times $[t_k,t_{k+1})$, and therefore, it may not hold in the continuous-time sense, which means that the forward invariance of the safe set $\mathcal{C}$ may not be guaranteed. 

Consider system \eqref{eqnsys} and a safe set $\mathcal{C}$ shown in \eqref{setc}. Suppose that the state measurement and actuation are perfect. The first problem that will be studied in this paper is stated as follows.

\begin{problem}\label{problem}
\emph{Design a sampled-data CBF-QP controller shown in \eqref{eqinput} that renders the set $\mathcal{C}$ forward invariant.}
\end{problem}


 The second problem considers system \eqref{eqnsys} with inaccurate state estimation and imperfect actuation. We assume that we only have access to an estimate $\hat x_k$ of the true system state $x_k$ such that $\hat x_k$ belongs to a bounded state uncertainty set centered at $x_k$. Similarly, we assume the real input produced by the imperfect actuator belongs to a bounded input uncertainty set centered at the desired input that is generated from a CBF-QP.

\begin{problem}\label{problem2}
\emph{With inaccurate state estimation and imperfect actuation, design a sampled-data CBF-QP controller shown in \eqref{eqinput}  that renders the set $\mathcal{C}$ forward invariant.}
\end{problem}

\section{Sampled-Data CBF Condition Based On Interval Analysis}\label{range_bound}

This section presents a framework based on interval analysis to solve Problem \ref{problem} with perfect state measurement and actuation.


\subsection{Interval Taylor Model of Nonlinear Functions}

Solving Problem \ref{problem} involves the computation of the range of functions using interval arithmetic. The simplest method is to directly apply interval arithmetic to each term of the function \cite{jaulin2001interval,moore2009introduction}; though fast, this method often results in rather conservative bounds.  Instead, the interval Taylor model will be utilized in this paper to obtain tighter bounds of the range of functions \cite{makino2003taylor,berz1998computation,makino2009rigorous}.


\begin{definition}\label{defTM}[Def. 1 in \cite{makino2003taylor}]
Let $f: S\subset\R^n\rightarrow \R$ be a function that is $(n+1)$ times continuously partially differentiable on an open set containing the domain $S$. Let $x_0$ be a point in $S$ and $P^n$ the $n$-th order
Taylor polynomial of $f$ around $x_0$. Let $I$ be an interval such that
\begin{align*}
f(x) \in P^n\left(x-x_{0}\right)+I,\;\forall x \in S.
\end{align*}
Then the pair $(P^n,I)$ is called an $n$-th order \emph{Interval Taylor model} of $f$ around $x_0$ on $S$. 
\end{definition} 

In general, the reminder interval $I$ will be smaller with a larger value of $n$.

\subsection{SDCBF with Relative Degree 1 }
The main idea of solving Problem \ref{problem} is to design a margin term $\phi(x_k,\Delta_k)$ that accounts for the difference between the continuously updated controller and the sampled-data controller, and add it to the CBF condition \eqref{eqsampled} such that the piecewise-constant controller as shown in \eqref{eqinput} can guarantee controlled invariance of set $\C$ in continuous time (i.e., $h(x(t))\geq 0$ for all $t\geq 0$ whenever $h(x(0))\geq 0$).  
Recall that $x_k = x(t_k)$. Given a CBF $h$ with relative degree 1, we call the following inequality
\begin{align}
L_f h(x_k) + L_g h(x_k) u_k + \gamma h(x_k) + \phi(x_k,\Delta_k) \geq 0\label{SDCBF}
\end{align}
the \emph{sampled-data CBF} (SDCBF) condition at sampling instance $t_k$.  Define the following set 
\begin{align}
\Kzcbf^s(x_k,\Delta_k) =  \{  u \in &U \mid L_f h(x_k) + L_g h(x_k) u \nonumber\\
&+ \gamma h(x_k) + \phi(x_k,\Delta_k)\geq 0\} \label{inputsetsample}
\end{align}
as the sampled-data admissible input set for the sampled state $x_k$ and the sampling interval $\Delta_k$.  Define a function 
\begin{align}\label{Deltaxi}
\Delta \xi(x,u,x_k)=\xi(x,u) - \xi(x_k,u)
\end{align}
where 
$$
\xi(\cdot,u)\triangleq L_{f} h(\cdot)+L_{g} h(\cdot) u+\gamma h(\cdot).  
$$
Define $z=(x^\top,u^\top)^\top$. For any given sampling interval $\Delta_k>0$, define the set $\mathcal{Z}_k$ as the Cartesian product of the reachable tube $\mathcal{R}([t_k,t_k+\Delta_k],x_k)$ and the admissible set of the input $U$, i.e.,
\begin{align}\label{eqz}
\mathcal{Z}_k\triangleq \mathcal{R}([t_k,t_k+\Delta_k],x_k)\times U.
\end{align}
Then we have the following result that solves Problem \ref{problem}.

\begin{proposition}\label{prop1}
Consider control system \eqref{eqnsys} and a set $\mathcal{C} \subset \R^n$ defined by \eqref{setc} for a $C^1$ function $h: \R^n \to \R$ that has a relative degree 1.  
Suppose that $z_k^*=({x_k^*}^{\top},{u_k^*}^{\top})^{\top}$ is a given state-input pair in the set $\mathcal{Z}_k$, i.e, $z_k^*\in \mathcal{Z}_k$, and $(P^n_{k},I_k)$ is the $n$-th Taylor model of $\Delta \xi(x,u,x_k)$ around $z_k^*$, i.e.,
\begin{align}
\Delta \xi(x,u,x_k)\in P^n_{k}(z-z_k^*)+I_k, \;\forall z \in \mathcal{Z}_k.\label{eqtaylor}
\end{align} 
Suppose that $\phi(x_k,\Delta_k)$ is chosen to be \begin{align}
\phi(x_k,\Delta_k)=\underline I_k+\phi^*_k \label{eqphi}    
\end{align} 
where $\underline I_k$ is the lower bound of $I_k$, $\phi^*_k=\min_{z\in \mathcal{Z}_k}P^n_{k}(z-z_k^*)$, and the resulting set $\Kzcbf^s(x_k,\Delta_k)$ is non-empty. If $h(x(0))\geq 0$, then any input $u(t)=u_k,t\in[t_k,t_k+\Delta_k)$ such that  $u_k\in \Kzcbf^s(x_k,\Delta_k)$ 
will render $h(x(t))\geq 0$ for all $t\geq 0$. 
\end{proposition}


\begin{proof} By  the definition of $\phi(x_k,\Delta_k)$ and the inclusion relation \eqref{eqtaylor},  $\Delta \xi(x,u,x_k)\geq \phi(x_k,\Delta_k)$ holds for any $z\in \mathcal{Z}_k$. 
For any $t\in [t_k, t_{k}+\Delta_k)$, since $\Delta \xi(x,u_k,x_k)=\xi(x,u_k) - \xi(x_k,u_k)$, it follows that 
\begin{align*}
&L_{f} h(x)+L_{g} h(x) u_k+\gamma h(x)\\
= &L_{f} h(x_k)+L_{g} h(x_k) u_k+\gamma h(x_k)+\Delta \xi(x,u_k,x_k)\\
\geq &L_{f} h(x_k)+L_{g} h(x_k) u_k+\gamma h(x_k)+\phi(x_k,\Delta_k) \geq 0
\end{align*}
where the last inequality is from the definition of $\Kzcbf^s(x_k,\Delta_k)$ shown in \eqref{inputsetsample} and the fact that $u_k\in \Kzcbf^s(x_k,\Delta_k)$. 
Therefore, by induction, for any $t\geq 0$, $L_{f} h(x)+L_{g} h(x) u+\gamma h(x)\geq 0$ holds, which implies that $h$ is a CBF for $\mathcal{C}$. Since $u(t)$ is piecewise cosntant and therefore locally Lipschitz, the conclusion holds immediately by Corollary 7 of \cite{Xu2015ADHS}. This completes the proof.
\end{proof}

Note that $z^*_k$ can be any element in the set $\mathcal{Z}_k$. In this paper, we will choose  $z_k^*=({x_k^*}^{\top},{u_k^*}^{\top})^{\top}=({x_k}^{\top},{u_c}^{\top})^{\top}$ where $u_c$ is the center of the input admissible set $U$. 

The sampled-data safe controller is obtained by solving the following (SDCBF-QP) only at discrete sampling times:
\begin{align}
u^*(x_k) &=    \underset{u\in U}{\operatorname{argmin}}   \quad \|u-u_{nom}\|_2 \label{eqn:SDCBFQP}\tag{SDCBF-QP}\\
\mathrm{s.t.}  &\; L_f h(x_k) + L_g h(x_k) u + \gamma h(x_k) +\phi(x_k,\Delta_k) \geq 0,\nonumber
\end{align}
where $k=0,1,...$ and $u_{nom}$ is any given nominal  controller.

Next, we consider how to compute the term $\phi(x_k,\Delta_k)$ in the SDCBF condition shown in \eqref{SDCBF} efficiently, whose value is needed to construct $\Kzcbf^s(x_k,\Delta_k)$ at each sampling time $t_k$. To enable real-time implementation of \eqref{eqn:SDCBFQP}, the value of $\phi(x_k,\Delta_k)$ needs to be obtained within $\Delta_k$ time.

For nonlinear systems the exact reachable set $\mathcal{R}([t_k,t_k+\Delta_k],x_k)$ is generally very challenging to compute, so we will use an over-approximation of $\mathcal{R}([t_k,t_k+\Delta_k],x_k)$, denoted as $\hat{ \mathcal{R}}([t_k,t_k+\Delta_k],x_k)$, to compute the Taylor model \eqref{eqtaylor} and the minimization \eqref{eqphi}. By replacing $\mathcal{R}([t_k,t_k+\Delta_k],x_k)$ with $\hat{ \mathcal{R}}([t_k,t_k+\Delta_k],x_k)$ in Proposition \ref{prop1}, the value of $\phi(x_k,\Delta_k)$ will be smaller which will render the admissible input set $\Kzcbf^s(x_k,\Delta_k)$ smaller; however, as long as $\Kzcbf^s(x_k,\Delta_k)$ is non-empty for every $k$, any input $u(t)=u_k,t\in[t_k,t_k+\Delta_k)$ such that  $u_k\in \Kzcbf^s(x_k,\Delta_k)$ 
will still guarantee the forward invariance of the set $\mathcal{C}$.  

The computation of $\phi(x_k,\Delta_k)$ in Proposition \ref{prop1} involves two tasks: 1) find $\hat{ \mathcal{R}}([t_k,t_k+\Delta_k],x_k)$ and 2) compute $\min_{z\in \mathcal{Z}_k}P^n_{k}(z-z_k^*)$. In the following, we will discuss how these two tasks can be accomplished efficiently.

1) \emph{Find $\hat{ \mathcal{R}}([t_k,t_k+\Delta_k],x_k)$.} 
To find $\hat{ \mathcal{R}}([t_k,t_k+\Delta_k],x_k)$, we will utilize the method in \cite{althoff2008reachability}, which is based on the linearization of a nonlinear system with interval remainder.
Consider a control system given in \eqref{eqnsys} and recall that $z=(x^\top,u^\top)^\top$. Given a state-input pair $z^*_k=({x^*_k}^\top,{u^*_k}^\top)^\top = $ $({x_k}^{\top},{u_c}^{\top})^{\top}$, the infinite Taylor series of the $i$-th state $x_i$ can be overapproximated by its first order Taylor polynomial and its Lagrange remainder as follows:
\begin{align*}
\dot{x}_{i}\in & F_{i}\left(x^{*}_k,u^*_k\right)+\left.\frac{\partial F_{i}(z)}{\partial z}\right|_{z=z^{*}_k}\left(z-z^{*}_k\right) + L_i([0,1])
\end{align*}
where 
\begin{align*}
L_i([0,1])&=\{\frac{1}{2}\left(z-z^{*}_k\right)^{\top} \frac{\partial^{2} F_{i}(z)}{\partial z^{2}}\left.\right|_{z=\zeta}\left(z-z^{*}_k\right)\mid\nonumber\\
&\quad \zeta=z^*_k+\theta(z-z^*_k),\theta\in [0,1]\}\in \IR
\end{align*}
and $z$ is restricted to a convex set. Therefore,  system \eqref{eqnsys} can be written into the following differential inclusion form:
\begin{align}
\dot x&\in F\left(z^{*}_k\right)+\left.\frac{\partial F(z)}{\partial z}\right|_{z=z^{*}_k}\left(z-z^{*}_k\right) + L([0,1])\nonumber\\
&=A(x-x^*_k)\!+\!B(u-u^*_k)\!+\!F\left(x^{*}_k,u^*_k\right)\! + \!L([0,1])\label{eqinclusion}
\end{align}
where 
\begin{align*}
A&=\left.\left(\frac{\partial f(x)}{\partial x}+\frac{\partial g(x)}{\partial x}u\right)\right|_{x=x^{*}_k,u=u^*_k}\in\R^{n\times n},\\
B&=g(x^*_k)\in\R^{n\times m},\\
L([0,1])&=[L_1([0,1]),\dots,L_n([0,1])]^\top\in \IR^n.
\end{align*}
The over-approximated reachable set $\hat{ \mathcal{R}}([t_k,t_k+\Delta_k],x_k)$ can be obtained by 
$$
\hat{ \mathcal{R}}([t_k,t_k+\Delta_k],x_k)=R_{lin}(x_k)\oplus R_{err}(x_k)
$$ 
where  $R_{lin}(x_k)$ is the over-approximated reachable set of the linearized system shown in \eqref{eqinclusion} with $L=0$, $R_{err}(x_k)$ is the over-approximated reachable set of the linearized system resulting from the remainder term $L$, and $\oplus$ denotes the Minkowski sum. As in \cite{althoff2008reachability}, we choose zonotopes or intervals as the representation of reachable sets because the computational efficiency of these representations. We utilize the same algorithms presented in \cite{althoff2008reachability} to compute $R_{lin}(x_k)$ and $R_{err}(x_k)$ for state $x_k$ at each  sampling time $t_k$. 




2) \emph{Compute $\min_{z\in \mathcal{Z}_k}P^n_{k}(z-z_k^*)$.} By the construction of $\hat{ \mathcal{R}}([t_k,t_k+\Delta_k],x_k)$ above, the set $\mathcal{Z}_k$ is represented as a zonotope or intervals, which can be readily expressed as a polytope, a more general set representation than zonotope/interval. Specifically, there exists a matrix $H_k$ and a column vector $\bf 1$ whose elements are all 1 with appropriate dimensions such that $\mathcal{Z}_k=\{(x,u)\mid H_k\begin{pmatrix}x\\u\end{pmatrix}\leq {\bf 1}\}$. 
Since $P^n_{k}(z-z_k^*)$ is a polynomial with variables $x$ and $u$,  the optimization problem $\min_{z\in \mathcal{Z}_k}P^n_{k}(z-z_k^*)$ becomes a polynomial optimization problem:
\begin{align*}
 \mbox{(POP)}\quad \phi^*_k = \min_{x,u} \quad & P^n_{k}\left(\begin{pmatrix}x\\u\end{pmatrix}-\begin{pmatrix}x_k^*\\u_k^*\end{pmatrix}\right)\\
  s.t.\quad  & H_k\begin{pmatrix}x\\u\end{pmatrix}\leq \bf 1
\end{align*}
A polynomial optimization problem is generally non-convex and  known to be NP-hard \cite{anjos2011handbook}. A polynomial optimization problem can be solved using non-convex global solvers, such as BMIBNB in YALMIP that is based on the branch \& bound algorithm. It also can be solved by relaxation methods either based on linear programming or semidefinite programming \cite{lasserre2002semidefinite}. In this paper we choose to solve (POP) by using Lasserre's linear matrix inequality (LMI) relaxations to obtain a lower bound for $\phi^*_k$, the global minimum of  (POP) \cite{lasserre2001global}. The relaxed LMIs in Lasserre's hierarchy are convex and can be solved using the interior-point algorithm in polynomial time, and the solutions of the LMIs provide a monotonically nondecreasing sequence of lower bounds for $\phi^*_k$.  Although the global optimal solution of (POP) can be obtained by increasing the relaxation order, the computational burden increases significantly with larger relaxation order. On the other hand, if $\phi(x_k,\Delta_k)$ is chosen to be $\phi(x_k,\Delta_k)=\underline I_k+\underline \phi_k$ where $\underline\phi_k$ is \emph{any} lower bound of $\phi_k^*$, then from the proof of Proposition \ref{prop1} it is easy to see that $u_k\in \Kzcbf^s(x_k,\Delta_k)$ will still render the set $\mathcal{C}$ forward invariant.

The following corollary formalizes  the discussion above.
\begin{corollary}\label{coro}
Suppose that $z_k^*\in \mathcal{Z}_k$, $(P^n_{k},I_k)$ is the $n$-th Taylor model of $\Delta \xi(x,u,x_k)$ around $z_k^*$, $\phi(x_k,\Delta_k)$ is chosen to be $\phi(x_k,\Delta_k)=\underline I_k+\underline \phi_k$ where $\underline I_k$ is the lower bound of $I_k$, $\underline \phi_k$ is the optimal value of any Lasserre's LMI for (POP), and  $\Kzcbf^s(x_k,\Delta_k)\neq \emptyset$ for every $k$. Then any input $u(t)=u_k,t\in[t_k,t_k+\Delta_k)$ such that  $u_k\in \Kzcbf^s(x_k,\Delta_k)$ 
will render the set $\mathcal{C}$ forward invariant.
\end{corollary}

We use SparsePOP  to exploit the sparse structure of polynomials when applying Lasserre's hierarchy of LMI relaxations to (POP) \cite{waki2008algorithm}, and use Mosek to solve the relaxed semidefinite programmings \cite{mosek}. The computational efficiency of finding $\hat{ \mathcal{R}}([t_k,t_k+\Delta_k],x_k)$ and computing $\min_{z\in \mathcal{Z}_k}P^n_{k}(z-z_k^*)$ will be demonstrated in simulations and experiments in Section \ref{simulation}.


\begin{remark}
In \cite{breeden2021control}, three types of modified CBF conditions for sampled-data systems were proposed. The computation of the CBF conditions there involves non-convex optimization problems that can be solved by nonlinear programming solvers such as FMINCON or IPOPT \cite{wachter2006implementation}. However, these solvers are sensitive to the initial conditions, have no guarantee on termination time in general and can only find local optimum values, which make them unsuitable for safety-critical applications. In addition, imperfect state estimation and actuation were not considered in \cite{breeden2021control}. In \cite{singletary2020control}, a robust backup controller-based CBF controller under state uncertainty is proposed requiring the sampled-data system to be incremental stable. Besides, the CBF condition in \cite{singletary2020control} involves the nonlinear robust optimization problems which might not be tractable for complex nonlinear dynamics.


Compared with existing results, the proposed framework is applicable to any nonlinear control affine dynamics. In particular, computing $\phi(x_k,\Delta_k)$ in \eqref{inputsetsample} is based on convex programs and has several advantages: (i) the related LMIs are convex programs that can be solved efficiently with real-time computation guarantees; (ii) by choosing the order of the Taylor polynomial and the relaxation order of Lasserre's LMI for (POP), we can make a trade-off between the computation's effiency and optimality; (iii) any lower bound of $\phi(x_k,\Delta_k)$ can be used to guarantee the safe set forward invariant as stated in Corollary \ref{coro}.
\end{remark}





\subsection{Extension to High Relative Degree Case}

Results in the preceding subsection can be easily generalized to CBF with a relative degree $r\geq 2$. 
For the sampled state $x_k$ and the sampling interval $\Delta_k$, define the sampled-data admissible input set as follows:
\begin{align}
&\Kzcbf^s(x_k,\Delta_k) =  \{  u \in U \mid L_f^rh(x_k)+L_g L_f^{r-1}h(x_k)u\nonumber\\
&\quad +{\bf a}^{\top}\eta(x_k)+ \phi(x_k,\Delta_k)\geq 0\}. \label{inputsetsample2}
\end{align}

\begin{proposition}\label{prop:highdegree}
Consider control system \eqref{eqnsys} and a set $\mathcal{C} \subset \R^n$ defined by \eqref{setc} for a $C^r$ function $h$ that is a CBF with a relative degree $r$ such that \eqref{ineq:ZCBF2} holds. Suppose that $z_k^*=({x_k^*}^{\top},{u_k^*}^{\top})^{\top}$ is a given state-input pair in the set $\mathcal{Z}_k$ where $\mathcal{Z}_k$ is given in \eqref{eqz}  and $(P^n_{k},I_k)$ is the $n$-th Taylor model of $\Delta \xi(x,u,x_k)$ around $z_k^*$ where $\Delta \xi(x,u,x_k)$ is 
defined as in \eqref{Deltaxi} with 
\begin{equation}\label{xidef2}
    \xi(\cdot,u)\triangleq L_g L_f^{r-1}h(\cdot)u +L_f^rh(\cdot)+{\bf a}^{\top}\eta(\cdot). 
\end{equation}
Suppose that $\phi(x_k,\Delta_k)$ is chosen to be $\phi(x_k,\Delta_k)=\underline I_k+\phi^*_k$ 
where $\underline I_k$ is the lower bound of $I_k$, $\phi^*_k=\min_{z\in \mathcal{Z}_k}P^n_{k}(z-z_k^*)$, and the resulting set $\Kzcbf^s(x_k,\Delta_k)$ is non-empty. If $s_k(x(0))\geq 0$ for $k=0,1,...,r-1$, where $s_k$ are given in \eqref{liftedh}, then any input $u(t)=u_k,t\in[t_k,t_k+\Delta_k)$ such that  $u_k\in \Kzcbf^s(x_k,\Delta_k)$ 
will render $h(x(t))\geq 0$ for all $t\geq 0$. 
\end{proposition}

\begin{proof}
Using the same proof procedure  as Proposition \ref{prop1}, one can show that any input $u(t)=u_k,t\in[t_k,t_k+\Delta_k)$ such that  $u_k\in \Kzcbf^s(x_k,\Delta_k)$ will render $L_g L_f^{r-1}h(x)u +L_f^rh(x)+{\bf a}^{\top}\eta(x)\geq 0$. Since $u(t)=u_k$ is piecewise constant, the roots of $p_0^r(\lambda)$ are all negative,  and $s_k(x(0))\geq 0$ for $k=0,1,...,r$,  the condition follows by the results of \cite{nguyen2016exponential}. 
\end{proof}

\section{Safety under Measurement \& Actuation Uncertainties}\label{uncertain}





In practice, the exact state information of a control system is unknown.  For sampled-data systems, an estimate of the system state is available at sampling instances, which can be obtained from an observer such as Luenberger or interval observer, or from a Kalman filter. The following assumption provides a measure of the estimation accuracy. 

\begin{assumption}\label{assumx}
At any time instance $t_k,k \in \mathbb{Z}_{\geq 0}$, the state of the system $x_k$ and the estimated state $\hat x_k$ satisfy
\begin{align*}
x_k \in \{\hat x_k\} \oplus B_{\epsilon_x}(0),
\end{align*}
where $B_{\epsilon_x}(0)$ is the 2-norm ball at the origin with a radius of $\epsilon_x >0$, i.e., $B_{\epsilon_x}(0)=\{x\in \mathbb{R}^n\ |\ \|x\|_2\leq \epsilon_x\}$.
\end{assumption}

 From Assumption \ref{assumx}, we have that $x_k \in B_{\epsilon_x}(\hat{x}_k)\triangleq\{x\in \mathbb{R}^n\ |\ \|x-\hat x_k\|_2\leq \epsilon_x\}$.
 Since we only have the estimated state of the system, we will guarantee the forward invariance of the set $\mathcal{C}$ utilizing the enlarged reachable tube $\mathcal{R}([t_k,t_k+\Delta_k],B_{\epsilon_x}(\hat{x}_k)) = \{x(t,x_0)\ |\ x_0 \in B_{\epsilon_x}(\hat{x}_k)),t\in[t_k,t_k+\Delta_k],u\in U\}$ and the corresponding set $\hat{\mathcal{Z}}_k$ defined as
\begin{equation}\label{eqzhat}
    \hat{\mathcal{Z}}_k \triangleq \mathcal{R}([t_k,t_k+\Delta_k],B_{\epsilon_x}(\hat{x}_k)))\times U.
\end{equation}

Besides the uncertainty from the state estimation, the real system might also have imperfect actuator which causes a deviation between the desired input and the real input. To account for the uncertain actuation, we need to bound this deviation and thus guarantee the system safety for the worst-case scenario.


\begin{assumption}\label{assumu}
At any time instance $t_k,k \in \mathbb{Z}_{\geq 0}$, the desired input $u^d_k$ and the real input $u^r_k$ implemented by the system satisfy:
\begin{align*}
u^r_k \in \{u^d_k\} \oplus B_{\epsilon_u}(0)
\end{align*}
where $B_{\epsilon_u}(0)=\{u\in \mathbb{R}^m\ |\ \|u\|_2\leq \epsilon_u\}$. 
\end{assumption}


Suppose that $\hat z_k^* = ({\hat x_k}^{\top},{u_c}^{\top})^{\top} \in \hat{\mathcal{Z}}_k$ where $u_c$ is the center of the set $U$ and $\hat{\mathcal{Z}}_k$ is given in \eqref{eqzhat}. $(\hat P^n_{k},\hat I_k)$ is the $n$-th Taylor model of $\Delta \xi(x,u,\hat x_k)$ around $\hat z_k^*$. Let $\phi(\hat x_k,\Delta_k)=\hat{\underline I}_k+\hat{\phi}^*_k$ where $\hat{\underline I}_k$ is the lower bound of $\hat I_k$, $\hat{\phi}^*_k=\min_{z\in \hat{\mathcal{Z}}_k}\hat P^n_{k}(z-\hat z_k^*)$.
Then, we define the new admissible input set as 
\begin{align}
    & \hat K_{\mathrm{zcbf}}^s (\hat x_k,\Delta_k)= \{u\in U\ominus B_{\epsilon_u}(0) \ |\ L_f^rh(\hat x_k) \nonumber \\ 
    & +L_g L_f^{r-1}h(\hat x_k)u +{\bf a}^{\top}\eta(\hat x_k) + \phi(\hat x_k,\Delta_k) \geq 0\}, \label{inputsetsample3}
\end{align}
where $\ominus$ is the Pontryagin difference.

The following result provides a solution to Problem \ref{problem2}.
\begin{proposition}\label{prop:uncertain}
Consider control system \eqref{eqnsys} and a set $\mathcal{C} \subset \R^n$ defined by \eqref{setc} for a $C^r$ function $h$ that is a CBF with a relative degree $r$ such that \eqref{ineq:ZCBF2} holds. 
Suppose $\hat K_{\mathrm{zcbf}}^s(\hat x_k,\Delta_k)$ is non-empty. If $\min_{x \in B_{\epsilon_x}(\hat x_0)} s_k(x)\geq 0$ for $k=0,1,...,r-1$, where $s_k$ are given in \eqref{liftedh}, then any input $u(t)=u_k^d,t\in[t_k,t_k+\Delta_k)$ such that  $u_k^d\in \hat K_{\mathrm{zcbf}}^s(\hat x_k,\Delta_k)$ will render $h(x(t))\geq 0$ for all $t\geq 0$. 

\end{proposition}

\begin{proof}
Since the desired input $u_{k}^d\in \hat K_{\mathrm{zcbf}}^s(\hat x_k,\Delta_k)$, according to Assumption \ref{assumu}, the real input $u_{k}^r\in \hat K_{\mathrm{zcbf}}^s(\hat x_k,\Delta_k) \oplus B_{\epsilon_u}(0)$ which implies that $u^r_k \in U$ and $L_f^rh(\hat x_k) +L_g L_f^{r-1}h(\hat x_k)u^r_k +{\bf a}^{\top}\eta(\hat x_k) + \phi(\hat x_k,\Delta_k) \geq 0$. Using the definition of $\phi(\hat x_k,\Delta_k)$, one can get that $ L_g L_f^{r-1}h(x)u +L_f^rh(x)+{\bf a}^{\top}\eta(x)\geq 0$ for all $t\geq 0$. Following the same proof as in Proposition \ref{prop:highdegree}, it is easy to show the forward invariance of the set $\mathcal{C}$.
\end{proof}

The sampled-data safe controller with inaccurate state estimation and imperfect actuator is obtained by solving the following uncertain sampled-data CBF-QP (USDCBF-QP) only at discrete sampling times:
\begin{align}
u^d(\hat x_k) &=    \underset{u\in U\ominus B_{\epsilon_u}(0)}{\operatorname{argmin}}   \quad \|u-u_{nom}\|_2 \label{eqn:USDCBFQP}\tag{USDCBF-QP}\\
\mathrm{s.t.}  &\; L_f h(\hat x_k) + L_g h(\hat x_k) u + \gamma h(\hat x_k) +\phi(\hat x_k,\Delta_k) \geq 0\nonumber
\end{align}
where $k=0,1,...$ and $u_{nom}$ is any given nominal  controller.

\section{Simulation \& Experiment}\label{simulation}

In this section, we demonstrate the effectiveness of the proposed USDCBF condition using two simulation examples and one experiment example on the Crazyflie Quadcopter \cite{PX4}. The sampled-data controller is used for all examples, but different CBF conditions are used in the QPs. For simplicity, we will refer to the sampled-data controller with the naive CBF constraint shown in \eqref{eqsampled} as CBF controller and the safe sampled-data controller by solving \eqref{eqn:USDCBFQP} as USDCBF controller. 

\begin{figure}[!b] 
\centering
\includegraphics[width=0.36\textwidth]{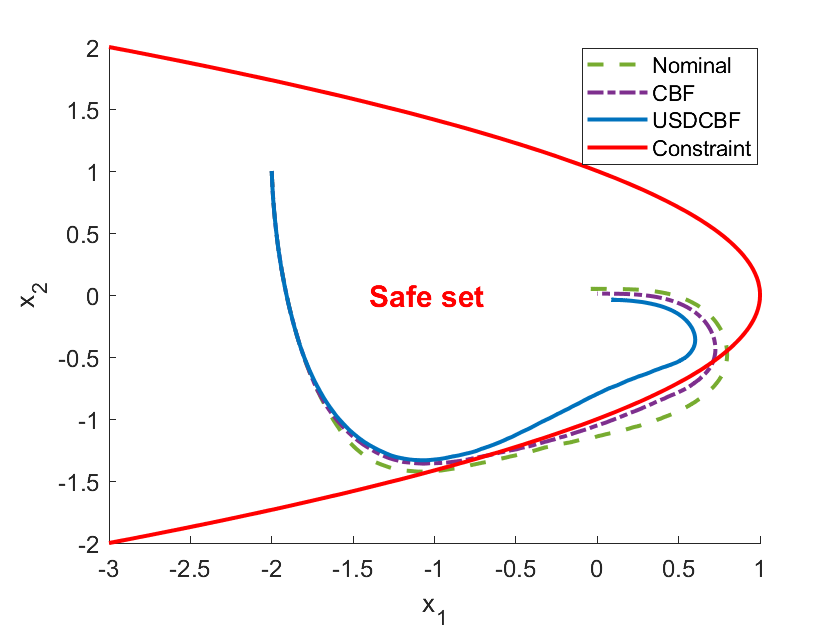}
\caption{Simulation result for Example \ref{exmp1}. System trajectories with three controllers are shown: (i) nominal controller (green), which is obtained from control Lyapunov function conditions; (ii) CBF controller (purple), which is a sampled-data CBF-QP controller with the naive CBF constraint shown in \eqref{eqsampled}; (iii) USDCBF controller (blue), which is a sampled-data controller obtained by solving (USDCBF-QP). The 0-level set of CBF $h$ as the safe constraint is shown in red. The measurement and actuation uncertainties are chosen as  $\epsilon_x = \epsilon_u = 0.1$.}
\label{exmp1_sim}
\end{figure}

\begin{example}\label{exmp1}

Consider the  following dynamics \cite{jankovic2018robust}:
\begin{align*}
    \dot{x}_{1} &=-0.6 x_{1}-x_{2} \\
    \dot{x}_{2} &=x_{1}^{3}+x_{2} u.
\end{align*}
Consider the safe set $\mathcal{C} = \{x\in \mathbb{R}^2: h(x)\geq 0\}$ where $h(x) = -x_2^2-x_1+1$, which has a relative degree 1. Assume the input is constrained in the set $U = \{u\ |-1\leq u\leq 1\}$, the periodic sampling time is 0.02 seconds and $\gamma = 3$ in \eqref{inputsetsample}. The uncertainty bounds on the estimation error and the actuation error are both 0.1, i.e., $\epsilon_x = \epsilon_u = 0.1$. The control objective is to steer the system to the origin while keeping the system trajectory inside the safe set $\mathcal{C}$. We choose the nominal controller to be a stabilizing controller based on control Lyapunov functions without considering the safety constraint and implement the sampled-data safe controller from \eqref{eqn:USDCBFQP}. The closed-loop system is simulated for 10 seconds starting from the initial state $x_0 = [-2,1]^{\top}$. The average computation time (including the computation of $\phi(x_k,\Delta_k)$ and solving \eqref{eqn:USDCBFQP}) at $t=t_0,t_1,...$ is around 0.018 seconds using MATLAB R2020b in a computer with 3.7 GHz CPU and 32 GB memory. Fig. \ref{exmp1_sim} shows system trajectories with the USDCBF controller and the CBF controller. It can be observed that in the presence of state measurement and actuation uncertainties, the CBF controller can not keep the system safe when the CBF condition is naively applied as in \eqref{eqsampled}; in contrast, the USDCBF controller respects the safety constraint for all time while steering the trajectory to the origin.

\end{example}

\begin{figure}[!b] 
\centering
\includegraphics[width=0.43\textwidth]{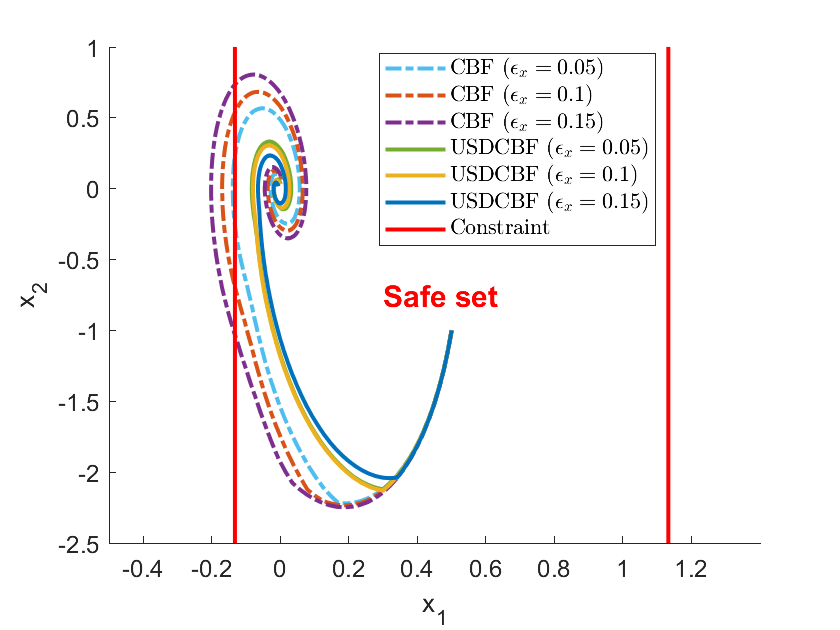} 
\caption{Simulation result for Example \ref{example2}. Trajectories of the mass-spring-damper system with the CBF controller and the USDCBF controller are shown for actuation uncertainty $\epsilon_x=0.1$ and three measurement uncertainties $\epsilon_x=0.05,0.1,0.15$. The sampled-data controller by solving the CBP-QP with constraint  \eqref{eqsampled} can not ensure safety, while the USDCBF-QP controller ensures safety for all three values of $\epsilon_x$. }\label{exmp2_x1_x2}
\end{figure}

\begin{example}\label{example2}
Consider the mass-spring-damper system: 
\begin{align*}
    \dot x_1 &= x_2,\\
    \dot x_2 &= -\frac{1}{m}(kx_1 + bx_2-u),
\end{align*}
where $k = 50$, $b = 3$ and $m = 1.5$. Consider the safe set $\mathcal{C} = \{x \in \R^2 : h(x) \geq 0\}$ where  $h(x) = 10 - \frac{k}{2}(x_1 - 0.5)^2$, which has a relative degree 2. This safe set corresponds to an upper and lower bound on the state $x_1$. 
The input set is $U = \{u\mid -10\leq u \leq 10\}$, the periodic sampling time is 0.02 seconds, the vector ${\bf a} = [20,100]^{\top}$ in Proposition \ref{prop:uncertain}, the nominal controller is $u_{nom}(x) = -\frac{3}{2}(x_1 + x_2)$ and the initial condition is $x_0 = [0.5,-1]^{\top}$. We choose the actuation uncertainty bound $\epsilon_u = 0.1$ and 
three different measurement uncertainty bounds $\epsilon_x=0.05,0.1,0.15$. Fig. \ref{exmp2_x1_x2} shows trajectories of the system with the CBF controller and the USDCBF controller for three values of $\epsilon_x$. It can be observed that in the presence of measurement and actuation uncertainties, trajectories with the CBF controller will violate the safety constraint, and the violation is larger as the state measurement uncertainty becomes larger. In contrast, trajectories with the USDCBF controller respect the safety constraint for any value of $\epsilon_x$. Also note that the USDCBF controller tends to be more conservative when the state measurement uncertainty $\epsilon_x$ is larger, which is expected. 
The average computation time at sampling times is around 0.02 seconds on the same computer as in Example \ref{exmp1}.

\end{example}

\begin{example}\label{example3}

This example presents the experimental results that implements the USDCBF controller on a Crazyflie Nano Quadcopter \cite{PX4}. 
We consider the following linearized 6-dimension quadcopter model:
\begin{equation}\label{model}
\dot{\x}=\left[\begin{array}{ll}
0_{3\times 3} & I_{3\times 3} \\
0_{3\times 3} & 0_{3\times 3}
\end{array}\right] \x+\left[\begin{array}{l}
0_{3\times 3}\\
I_{3\times 3}
\end{array}\right] \ub.    
\end{equation}
where $\x=[x\quad y\quad z\quad \dot{x}\quad \dot{y}\quad \dot{z}]^{T}$ is the state and  $\ubold = [\ddot{x}\quad \ddot{y}\quad \ddot{z}]^{\top}$ is the virtual input.
The model shown in \eqref{model} is usually referred to as the double-integrator quadcopter model and is widely used in quadcopter simulations \cite{xu2018safe,greeff2018flatness}. 
Define the safe set $\mathcal{C} = \{\x \in \R^6 : \mathbf{h}(\x) \geq \mathbf{0}\}$ with
\begin{align*}
    \mathbf{h}(\x) & = [h_1(\x),\; h_2(\x),\; h_3(\x),\;h_4(\x),\; h_5(\x),\; h_6(\x) ]^{\top} \\
    & = [\bar x -x ,\; x-\underline x ,\; \bar y -y ,\; y-\underline y ,\; \bar z -z ,\; z-\underline z]^{\top}
\end{align*}
where $\bar x = 0.5,\ \bar y=0.5,\ \bar z = 0.6$ and $\underline x=-0.5,\ \underline y=-0.5,\ \underline z = 0$ are the upper bounds and lower bounds on the position of the quadcopter respectively. It's easy to check that $h_i(\x),\; i=1,..,6$, are all CBFs with relative degree 2. Therefore, the SDCBF condition for $i=1,...,6$ is given by 
$
L_f^2 h_i(\x)+L_g L_f h_i(\x) \ub + \mathbf{a}^{\top} \eta_i (\x) + \phi_i(\x,\Delta t) \geq 0,\;
$ where $\eta_i (\x)= [L_f h_i(\x), h_i(\x)]^{\top}$ and $\mathbf{a}= [6, 8]^{\top}$. We use a linear quadratic regulator controller as the nominal controller to track a given reference trajectory and apply the SDCBF condition to keep the quadcopter in the safe set $\mathcal{C}$. 


We choose $\epsilon_x = 0.02$ and $\epsilon_u = 0.01$. The quadcopter flies for about 20 seconds to follow the reference trajectory starting from the origin.  We implement both CBF and USDCBF controllers in the quadcopter experiments, with the frequency of the control input signal set to 50Hz and 100Hz (the periodic sampling time is 0.02 and 0.01 seconds respectively). 
Fig. \ref{exmp3} illustrates the reference trajectory and the quadcopter trajectories with two types of CBF-based controllers. Although most of the trajectory using the CBF controller is inside the constraining box (the safe set $\mathcal{C}$), $\mathbf{h}(\x)\geq \mathbf{0}$ is violated at some extreme points in Fig. \ref{exmp3-2} at frequency 50Hz. In contrast, the trajectory using the USDCBF controller remains in the constraining box for both frequencies for all time. The experimental results show that the USDCBF controller can guarantee safety under measurement and actuation uncertainties, which is necessary for the quadcopter and other safety-critical robotic applications. In this example, because $h_i (i=1,...,6)$ and the system dynamics are both linear, the polynomial optimization problem becomes a linear optimization problem which can be efficiently solved by linear programming solvers via simplex or interior-point methods. The average computation time is 0.005 seconds at each sampling time.


  %

\begin{figure}[!t]
\centering
\begin{subfigure}[b]{0.3\textwidth}
    \includegraphics[width=0.99\textwidth]{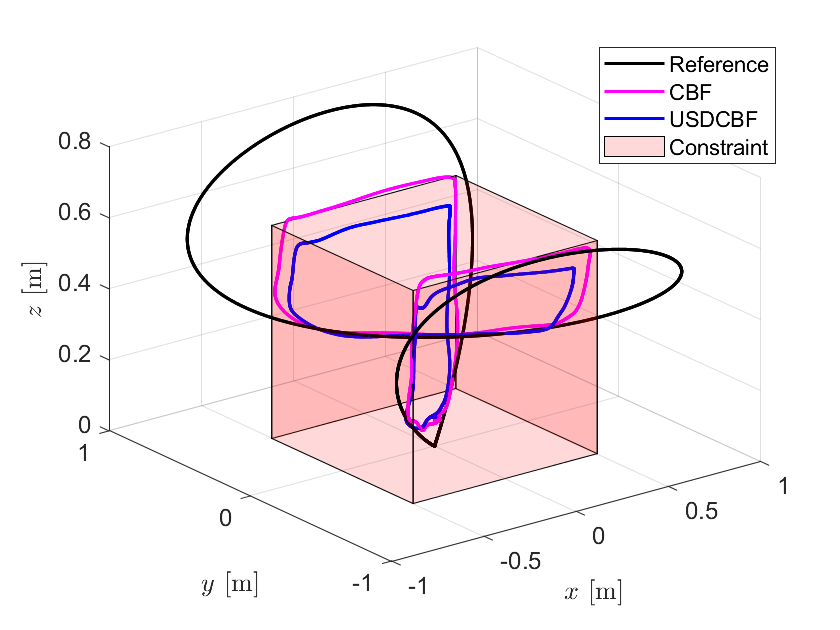}
    \caption{Experimental results of Example \ref{example3} when controller frequency is 50 Hz}
    \label{quad3d50}
  \end{subfigure}
  \begin{subfigure}[b]{0.3\textwidth}
    \includegraphics[width=0.99\textwidth]{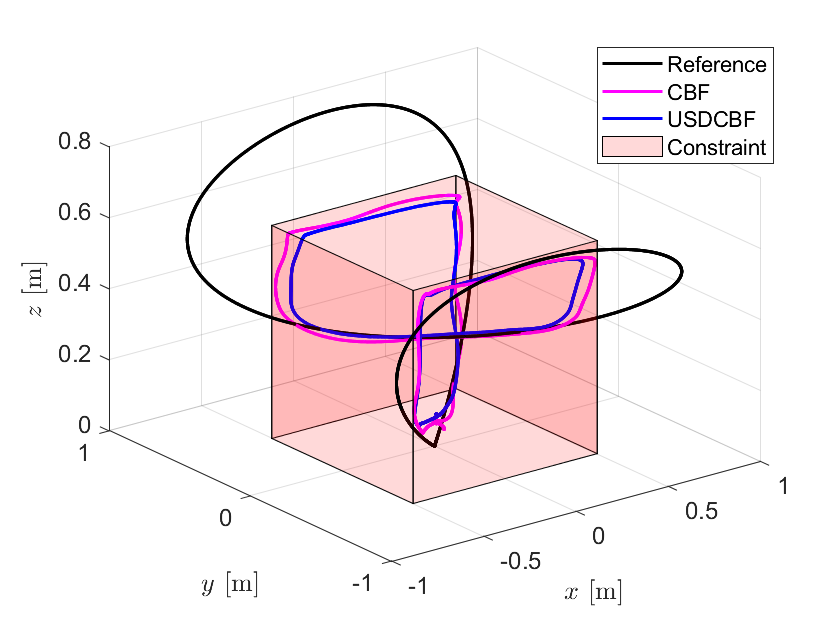}
    \caption{Experimental results of Example \ref{example3} when controller frequency is 100 Hz}
    \label{quad3d100}
  \end{subfigure}
  \caption{Experimental results for Example \ref{example3}. Trajectories of Crazyflie with the CBF controller (pink) and the USDCBF controller (blue) are shown with measurement and actuation uncertainties $\epsilon_x=0.02$ and $\epsilon_u=0.01$ for two controller frequencies. Reference trajectories are shown in black.}\label{exmp3}
\end{figure}

\begin{figure}[!ht]
\centering
    \includegraphics[width=0.43\textwidth]{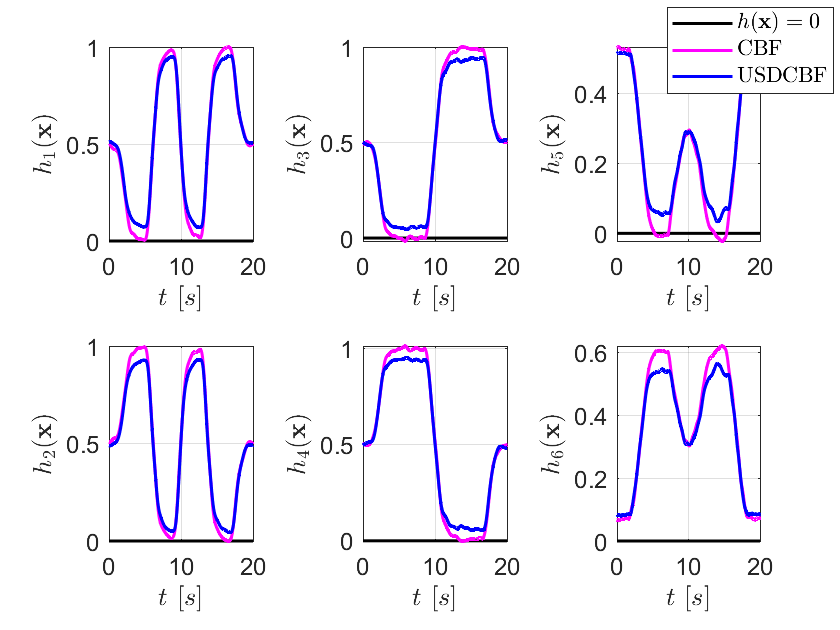}
  \caption{Evolution of CBF $\mathbf{h}(\x)$ in Example \ref{example3} at 50Hz. The USDCBF-QP controller respects all safety constraints while the CBF controller induces safety violations.}\label{exmp3-2}
\end{figure}

  %
  

\end{example}

\section{Conclusion} \label{conclude}

In this paper, we proposed a framework that can guarantee the safe control of sampled-data systems with measurement and actuation uncertainties. Comparing with the traditional CBF condition for continuous-time systems, the proposed SDCBF condition includes an additional term $\phi(x_k,\Delta_k)$ which can be efficiently solved by computing the lower bound of a Taylor polynomial using reachable tube approximation and polynomial optimization techniques. We proved that the SDCBF-QP controller can guarantee the safety constraint in continuous time for sampled-data systems with perfect information. We also showed that the USDCBF-QP controller can ensure safety with inaccurate state estimation and imperfect actuation. The performance of the proposed method was demonstrated via simulations and hardware experiments on the quadcopter. Future work includes developing more efficient and less conservative methods for reachable tube approximation by utilizing parallel computation and applying the proposed framework to more robotic applications.
\bibliographystyle{IEEEtran}
\bibliography{sampledref,intervalbib,reachbib}

\end{document}